\documentclass[12pt]{article}
\usepackage{amsmath}
\usepackage{amssymb}
\usepackage{amsthm}

\newtheorem{lemma}{Lemma}
\newtheorem{theorem}{Theorem}

\newcommand{\dist}{\mbox{\em dist\/}}

\newenvironment{proofof}[1]{\bigskip\noindent{\itshape #1. }}{\hfill$\Box$\medskip}

\makeatletter
\newcounter{subclaim}
\newenvironment{subclaim}{\stepcounter{subclaim}\begin{itemize}\item[(\roman{subclaim})]\def\@currentlabel{(\roman{subclaim})}}{\end{itemize}}
\let\@oldproof=\proof\def\proof{\setcounter{subclaim}{0}\@oldproof} % reset counter at beginning of proof
\makeatother

\begin{document}

\begin{center}
{\bf\Large A De Bruijn--Erd\H os theorem\\ for chordal graphs}\\
\vspace{0.5cm}
Laurent Beaudou (Universit\' e Blaise Pascal, Clermont-Ferrand)\footnote{{\tt laurent.beaudou@ens-lyon.org}}\\
 Adrian Bondy (Universit\' e Paris 6)\footnote{{\tt adrian.bondy@sfr.fr}}\\
 Xiaomin Chen (Shanghai Jianshi LTD)\footnote{{\tt gougle@gmail.com}}\\
 Ehsan Chiniforooshan (Google, Waterloo)\footnote{{\tt chiniforooshan@alumni.uwaterloo.ca}}\\
 Maria Chudnovsky (Columbia University, New York)\footnote{{\tt mchudnov@columbia.edu}\\
\hphantom{lllllxxx}Partially supported by NSF grants DMS-1001091 and IIS-1117631}\\
 Va\v sek Chv\' atal (Concordia University, Montreal)\footnote{{\tt chvatal@cse.concordia.ca}\\
\hphantom{lllllxxx}Canada Research Chair in Combinatorial Optimization}\\
 Nicolas Fraiman (McGill University, Montreal)\footnote{{\tt nfraiman@gmail.com}}\\
 Yori Zwols (Concordia University, Montreal)\footnote{{\tt yzwols@gmail.com}}
\end{center}

\begin{center}
{\bf Abstract}
\end{center} 
\vspace{-0.3cm}
{\small A special case of a combinatorial theorem of De
  Bruijn and Erd\H os asserts that every noncollinear set of $n$
  points in the plane determines at least $n$ distinct lines. Chen and
  Chv\'{a}tal suggested a possible generalization of this assertion in
  metric spaces with appropriately defined lines. We prove this
  generalization in all metric spaces induced by connected chordal
  graphs.}

\section{Introduction}\label{sec.intro}

It is well known that 
\begin{subclaim} \label{claim.dbe}
{\em
every noncollinear set of $n$ points in the plane\\ 
determines at least $n$ distinct lines.\/}
\end{subclaim}
As noted by Erd\H{o}s~\cite{E43}, theorem~\ref{claim.dbe} is a corollary of the
Sylvester--Gallai theorem (asserting that, for every noncollinear set
$S$ of finitely many points in the plane, some line goes through
precisely two points of $S$); it is also a special case of a
combinatorial theorem proved later by De Bruijn and Erd\H os
\cite{DE48}.\\

Theorem \ref{claim.dbe} involves neither measurement of distances nor
measurement of angles: the only notion employed here is incidence of
points and lines. Such theorems are a part of {\em ordered geometry\/}
\cite{C61}, which is built around the ternary relation of {\em
  betweenness\/}: point $b$ is said to lie between points $a$ and $c$
if $b$ is an interior point of the line segment with endpoints $a$ and
$c$. It is customary to write $[abc]$ for the statement that $b$ lies
between $a$ and $c$. In this notation, a {\em line\/} $\overline{uv}$
is defined --- for any two distinct points $u$ and $v$ --- as
\begin{equation}\label{def.first}
\{u,v\}\;\cup\; \{p: [puv] \vee [upv]\vee [uvp]\}. 
\end{equation}

In terms of the Euclidean metric $\dist$, we have
\begin{multline}\label{def.btw}
\mbox{ $[abc] \;\Leftrightarrow\;$}\\
\mbox{ $a,b,c$ are three distinct points and
$\dist(a,b)+\dist(b,c) = \dist(a,c)$.}
\end{multline}
In an arbitrary metric space, equivalence (\ref{def.btw}) defines the
ternary relation of {\em metric betweenness} introduced in \cite{Men}
and further studied in \cite{Blu,Bus,Chv}; in turn, (\ref{def.first})
defines the line $\overline{uv}$ for any two distinct points $u$ and
$v$ in the metric space. The resulting family of lines may have
strange properties. For instance, a line can be a proper subset of
another: in the metric space with points $u,v,x,y,z$ and 
\begin{align*}
&\dist(u,v)=\dist(v,x)=\dist(x,y)=\dist(y,z)=\dist(z,u)=1,\\
&\dist(u,x)=\dist(v,y)=\dist(x,z)=\dist(y,u)=\dist(z,v)=2,
\end{align*}
we have
\[
\overline{vy}=\{v,x,y\} \;\;\mbox{ and }\;\;
\overline{xy}=\{v,x,y,z\}. 
\]

\bigskip 

Chen \cite{Che} proved, using a definition of $\overline{uv}$
different from~(\ref{def.first}), that the Sylvester--Gallai theorem
generalizes in the framework of metric spaces. Chen and Chv\'{a}tal
\cite{CC08} suggested that theorem \ref{claim.dbe}, too, might
generalize in this framework:
\begin{subclaim} \label{ccc}
{\em True or false? Every metric space on $n$ points, where $n\ge 2$,
  either has at least $n$ distinct lines or else has a line that
  consists of all $n$ points.}
\end{subclaim}
They proved that  
\begin{itemize}
\item every metric space on $n$ points either has at least $\lg n$
  distinct lines or else has a line that consists of all $n$ points
\end{itemize}
and noted that the lower bound $\lg n$ can be improved to $\lg n +
\frac{1}{2}\lg\lg n + \frac{1}{2}\lg\frac{\pi}{2}- o(1)$.\\

Every connected undirected graph induces a metric space on its vertex
set, where $\dist(u,v)$ is defined as the smallest number of edges in
a path from vertex $u$ to vertex $v$.  Chiniforooshan and Chv\'{a}tal
\cite{CC11} proved that
\begin{itemize}
\item every metric space induced by a connected graph on $n$
  vertices either has $\Omega(n^{2/7})$ distinct lines or else has a 
  line that consists of all $n$ vertices;
\end{itemize}
we will prove that the answer to \ref{ccc} is `true' for all metric
spaces induced by connected chordal graphs.
\begin{theorem}\label{main}
  Every metric space induced by a connected chordal graph on $n$
  vertices, where $n\ge 2$, either has at least $n$ distinct lines or
  else has a line that consists of all $n$ vertices.
\end{theorem}
For graph-theoretic terminology, we refer the reader to Bondy and
Murty\cite{BM}.

\section{The proof}\label{sec.proof}

Given an undirected graph, let us write $[abc]$ to mean that $a,b,c$
are three distinct vertices such that
$\dist(a,b)+\dist(b,c)=\dist(a,c)$; this is equivalent to saying that
$b$ is an interior vertex of a shortest path from $a$ to $c$.

\begin{lemma}\label{lem1}
  Let $s,x,y$ be vertices in a finite chordal graph such that
  $[sxy]$. If $\overline{sx}= \overline{sy}$, then $x$ is a cut vertex
  separating $s$ and $y$.
\end{lemma}
\begin{proof}
  The set of all vertices $u$ such that $\dist(s,u)=\dist(s,x)$
  separates $s$ and $y$. Among all its subsets that separate $s$ and
  $y$, choose a minimal one and call it $C$. Since $x$ is an interior
  vertex of a shortest path from $s$ to $y$, it belongs to $C$. To
  prove that $C$ includes no other vertex, assume, to the contrary,
  that $C$ includes a vertex $u$ other than $x$.

\medskip

  Our graph with $C$ removed has distinct connected components $S$ and
  $Y$ such that $s\in S$ and $y\in Y$; the minimality of $C$
  guarantees that each of its vertices has at least one neighbour in
  $S$ and at least one neighbour in $Y$. Since each of $u$ and $x$ has
  at least one neighbour in $S$, there is a path from $u$ to $x$ with
  at least one interior vertex and with all interior vertices in $S$.
  Let $P$ be a shortest such path; note that $P$ has no chords except
  possibly the chord $ux$. Similarly, there is a path $Q$ from $u$ to
  $x$ with at least one interior vertex, and with all interior
  vertices in $Y$, that has no chords except possibly the chord
  $ux$. The union of $P$ and $Q$ is a cycle of length at least four;
  since this cycle must have a chord, vertices $u$ and $x$ must be
  adjacent. In turn, the union of $Q$ and $ux$ is a chordless cycle,
  and so $Q$ has precisely two edges. This means that some vertex $v$
  in $Y$ is adjacent to both $u$ and $x$.

\medskip

  Write $i=\dist(s,x)$ and $j=\dist(x,y)$.  Since all vertices $t$
  with $\dist(s,t)<i$ belong to $S$ and since $v$ has no neighbours in
  $S$, we must have $\dist(s,v)>i$; since $\dist(x,v)=1$, we conclude
  that $\dist(s,v)=i+1$ and that $v\in \overline{sx}$. Since
  $\overline{sx}=\overline{sy}$, it follows that $v\in \overline{sy}$.
  Since $\dist(v,x)=1$ and $\dist(x,y)=j$, we have
  $\dist(v,y)\le j+1$. From $\dist(s,v)=i+1$, $\dist(s,y)=i+j$,
  $\dist(v,y)\le j+1$, $i\ge 1$, $j\ge 1$, and $v\in \overline{sy}$, we
  deduce that $\dist(v,y)=j-1$.

\medskip

  Since $\dist(u,v)=1$, it follows that $\dist(u,y)\le j$; since
  $\dist(s,u)=i$ and $\dist(s,y)=i+j$, we conclude that $\dist(u,y)=j$
  and $u\in \overline{sy}$. Since $\dist(s,u)=i$, $\dist(s,x)=i$, and
  $\dist(u,x)=1$, we have $u\not\in \overline{sx}$. But then
  $\overline{sx}\ne \overline{sy}$, a contradiction.
\end{proof}

A vertex of a graph is called {\em simplicial\/} if its neighbours are pairwise adjacent.
\begin{lemma}\label{lem2}
  Let $s,x,y$ be three distinct vertices in a finite connected chordal
  graph. If $s$ is simplicial and $\overline{sx}= \overline{sy}$, then
  $\overline{xy}$ consists of all the vertices of the graph.
\end{lemma}
\begin{proof}
Since $\overline{sx}= \overline{sy}$, we have $y\in\overline{sx}$, and
so $[ysx]$ or $[syx]$ or $[sxy]$; since $s$ is simplicial, $[ysx]$ is
excluded; switching $x$ and $y$ if necessary, we may assume that
$[sxy]$. Given an arbitrary vertex $u$, we have to prove that
$u\in\overline{xy}$. Let $P$ be a shortest path from $s$ to $u$ and
let $Q$ be a shortest path from $u$ to $y$. Lemma~\ref{lem1}
guarantees that $x$ is a cut vertex separating $s$ and $y$, and so the
concatenation of $P$ and $Q$ must pass through $x$. This means that
$[sxu]$ or $[uxy]$ (or both). If $[uxy]$, then $u\in\overline{xy}$; to
complete the proof, we may assume that $[sxu]$, and so
$u\in\overline{sx}$.

\medskip

Since $\overline{sx}=\overline{sy}$, we have $[usy]$ or $[suy]$ or
$[syu]$; since $s$ is simplicial, $[usy]$ is excluded.  If $[suy]$, then
$[sxu]$ implies $[xuy]$; if $[syu]$, then $[sxy]$ implies $[xyu]$; in
either case, $u\in\overline{xy}$.
\end{proof}

\begin{proofof}{Proof of Theorem~\ref{main}} Consider a connected chordal 
graph on $n$ vertices where $n\ge 2$. By a theorem of Dirac~\cite{D},
this graph has at least two simplicial vertices; choose one of them
and call it $s$.  We may assume that the lines $\overline{sz}$ with
$z\ne s$ are pairwise distinct (else some line consists of all $n$
vertices by Lemma~\ref{lem2}). Since the graph is connected and has at
least two vertices, $s$ has at least one neighbour; choose one and
call it $u$. If $u$ is the only neighbour of $s$, then every path from
$s$ to another vertex must pass through $u$, and so $\overline{su}$
consists of all $n$ vertices.  If $s$ has a neighbour $v$ other than
$u$, then line $\overline{uv}$ is distinct from all of the $n-1$ lines
$\overline{sz}$ with $z\ne s$: since $s,u,v$ are pairwise adjacent, we
have $s\not\in\overline{uv}$.
\end{proofof}

\section{Related theorems}\label{sec.conc}

In Theorem~\ref{main}, `connected chordal graph' can be replaced by `connected bipartite graph':
\begin{itemize}
\item   every metric space induced by a connected bipartite graph on $n$
  vertices, where $n\ge 2$, has a line that consists of all $n$ vertices.
\end{itemize}

In fact, $\overline{xy}$ consists of all $n$ vertices whenever $x$ and
$y$ are adjacent. To prove this, consider an arbitrary vertex
$u$. Since the graph is bipartite, $\dist(u,x)$ and $\dist(u,y)$ have
distinct parities; since $\dist(x,y)=1$, they differ by at most
one. We conclude that $\dist(u,x)$ and $\dist(u,y)$ differ by
precisely one, and so $u\in\overline{xy}$.

\medskip

In Theorem~\ref{main}, `connected chordal graph' can be also replaced
by `sufficiently large graph of diameter two': Chiniforooshan and
Chv\'{a}tal \cite{CC11} proved that
\begin{itemize}
\item every metric space on $n$ points where each nonzero distance
  equals $1$ or $2$ has $\Omega(n^{4/3})$ distinct lines and
  this bound is tight.
\end{itemize}

\bigskip

{\bf\Large Acknowledgment}

\bigskip

\noindent The work whose results are reported here began at a workshop
held at Concordia University in June 2011.  We are grateful to the
Canada Research Chairs program for its generous support of this
workshop. We also thank Luc Devroye, Fran\c cois Genest, and Mark
Goldsmith for their participation in the workshop and for stimulating
conversations.


\begin{thebibliography}{}

\bibitem{Blu} L.M. Blumenthal, Theory and Applications of Distance
Geometry, Oxford University Press, Oxford, 1953.

\bibitem{BM} J.A.~Bondy, and U.S.R.~Murty, Graph Theory, Springer, New
  York, 2008.

\bibitem{Bus}
H. Busemann, The Geometry of Geodesics, Academic Press, New York,
1955.

\bibitem{Che} X. Chen, The Sylvester--Chv\'{a}tal theorem, {\em Discrete \&
Computational Geometry\/} {\bf 35} (2006), 193--199.

\bibitem{CC08} X. Chen and V. Chv\'{a}tal, Problems related to a de
  Bruijn--Erd\H os theorem, {\em Discrete Applied Mathematics\/} {\bf 156} (2008), 
  2101--2108.

\bibitem{CC11} E. Chiniforooshan and V. Chv\'{a}tal, A de
  Bruijn--Erd\H{o}s theorem and metric spaces, {\em Discrete
    Mathematics \& Theoretical Computer Science\/} {\bf 13} (2011),
  67--74.

\bibitem{C61}
H.S.M. Coxeter, Introduction to Geometry, Wiley, New York, 1961.

\bibitem{Chv} V. Chv\'{a}tal, Sylvester--Gallai theorem and metric
betweenness, {\em Discrete \& Computational Geometry\/} {\bf 31} (2004), 175--195.

\bibitem{D} G.A. Dirac, On rigid circuit graphs, {\em
    Abh. Math. Sem. Univ. Hamburg\/} {\bf 25} (1961), 71--76.

\bibitem{DE48}
N.G. De Bruijn and P. Erd\H os, On a combinatorial problem, 
{\em Indagationes Mathematicae\/}  {\bf 10} (1948), 421--423. 

\bibitem{E43}
P. Erd\H os, Three point collinearity, {\em American Mathematical
Monthly\/} {\bf 50} (1943), Problem 4065, p.~65. Solutions in
Vol. {\bf 51} (1944), 169--171.

\bibitem{Men}
K. Menger, Untersuchungen \"uber allgemeine Metrik,
{\em Mathematische Annalen\/} {\bf 100} (1928), 75--163. 
 
\end{thebibliography}
\end{document}